\documentclass[reqno]{amsart}
\usepackage{amsmath, amsthm, amssymb, amstext,amsaddr}

\usepackage{hyperref,xcolor}
\hypersetup{pdfborder={0 0 0},colorlinks}
\usepackage{enumitem}
\allowdisplaybreaks
\usepackage{graphicx}
\usepackage{caption}
\usepackage{mathtools}
\usepackage{tikz}

\newtheorem{theorem}{Theorem}
\newtheorem{remark}[theorem]{Remark}
\newtheorem{lemma}[theorem]{Lemma}
\newtheorem{proposition}[theorem]{Proposition}

\newtheorem{definition}[theorem]{Definition}
\newtheorem{example}[theorem]{Example}

\DeclareMathOperator*{\loc}{loc}
\DeclareMathOperator*{\dx}{d\textit{x}}

\DeclareMathOperator*{\dr}{d\textit{r}}

\newcommand{\R}{\mathbb{R}}
\newcommand*\diff{\mathop{}\!\mathrm{d}}
\newcommand{\h}{\ensuremath{\mathcal{H}}}
\newcommand{\hlog}{\ensuremath{\mathcal{H}_{\log}}}
\newcommand{\Lp}[1]{L^{#1}(\Omega)}

\newcommand{\Wp}[1]{W^{1,#1}(\Omega)}

\newcommand{\Wpzero}[1]{W^{1,#1}_0(\Omega)}

\newcommand{\ph}{\varphi}
\newcommand{\into}{\int_{\Omega}}
\newcommand{\weak}{\rightharpoonup}

\newcommand{\close}{\overline{\Omega}}

\newcommand{\WHzero}{W^{1, \h}_0(\Omega)}

\newcommand{\WHlogzero}{W^{1, \hlog}_0(\Omega)}

\newcommand{\Wph}{W^{1, \ph}(\Omega)}
\newcommand{\Wphzero}{W^{1, \ph}_0(\Omega)}
\newcommand{\Lph}{L^{\ph}(\Omega)}

\def\abs#1{\left|{#1}\right|}
\def\norm#1{\left\|#1\right\|}

\def\normph#1{\ensuremath{\left\|#1\right\|_{\ph}}}
\def\modph#1{\ensuremath{\varrho_{\ph} \left(#1\right)}}
\def\normoneph#1{\ensuremath{\left\|#1\right\|_{1,\ph}}}

\def\normonephzero#1{\ensuremath{\left\|#1\right\|_{1,\ph,0}}}

\numberwithin{theorem}{section}
\numberwithin{equation}{section}

\title[Monotonicity formulas and \texorpdfstring{(S$_+$)}{(S+)}-property]{Monotonicity formulas and \texorpdfstring{(S$_+$)}{(S+)}-property:\\ old and new}

\author[\'{A}ngel Crespo-Blanco]{\'{A}ngel Crespo-Blanco*}
\address[\'{A}ngel Crespo-Blanco]{Technische Universit\"{a}t Berlin, Institut f\"{u}r Mathematik\\ Stra\ss e des 17.\,Juni 136, 10623 Berlin, Germany}
\email{crespo@math.tu-berlin.de}


\thanks{*:Corresponding author.}

\begin{document}

\begin{abstract}
	The connection between monotonicity formulas and the \texorpdfstring{(S$_+$)}{(S+)}-property is that, for some popular differential operators, the former is used to prove the latter. The purpose of this paper is to explore this connection, remark how in the past both the monotonicity formulas and the \texorpdfstring{(S$_+$)}{(S+)}-property were focused on power-law growth, and prove the same type of results for a more general class of operators.
\end{abstract}

\maketitle

\section{Power-law case}\label{Power-law_case}

Monotonicity formulas have been a crucial tool for the study of partial differential equations and calculus of variations during the last half of a century. The reason behind it is that, to study equations which are not linear (or functionals which are not quadratic), this tool allowed to adapt arguments from the linear equations to these new settings. The other side of this story is that only a specific kind of nonlinearity was considered, that is, those with growth of a power-law, which materialized in equations driven by the $p$-Laplacian operator. The most popular version of these inequalities can be found in classical texts like \cite[equation (2.2)]{Simon-1978} or \cite[Chapter 12]{Lindqvist-2019}. For completeness of this text, we also include it in the following lines. The proof is based on the arguments of \cite[equation (2.2)]{Simon-1978} or \cite[Chapter 12]{Lindqvist-2019}.

\begin{lemma}
	\label{Le:MonotoneInequality}
	Let $r > 1$, for any $\xi, \eta \in \R^N$,
	\begin{align*}
		\left( \abs{\xi}^{r-2} \xi - \abs{\eta}^{r-2} \eta \right)
		\cdot \left( \xi - \eta \right) \geq
		C_r \abs{ \xi - \eta }^r 
	\end{align*}
	if $r \geq 2$, and
	\begin{align*}
		\left( \abs{\xi} + \abs{\eta} \right)^{2 - r} \left(  \abs{\xi}^{r-2} \xi - \abs{\eta}^{r-2} \eta \right)
		\cdot \left( \xi - \eta \right)  \geq
		C_r \abs{\xi - \eta}^2 
	\end{align*}
	if $1 < r < 2$, where 
	\begin{align*}
		C_r =
		\begin{cases}
			\min \{ 2^{2-r}, 2^{-1} \} & \text{if } r \geq 2,  \\
			r-1                        & \text{if } 1 < r < 2.
		\end{cases}
	\end{align*}
	The constant $C_r$ might not be optimal.
\end{lemma}
\begin{proof}
	For the case $r \geq 2$, we obtain the identity
	\begin{align*}
		& \left( \abs{\xi}^{r-2} \xi - \abs{\eta}^{r-2} \eta \right)
		\cdot \left( \xi - \eta \right) \\
		& = \left(  \abs{\xi}^{r-2} \xi \right)
		\cdot \left( \xi - \eta \right) +
		\left( - \abs{\eta}^{r-2} \eta \right)
		\cdot \left( \xi - \eta \right) \\
		& = \abs{\xi}^{r-2} \left( \frac{1}{2} \xi - \frac{1}{2} \eta \right) \cdot \left( \xi - \eta \right)
		+ \abs{\eta}^{r-2} \left( \frac{1}{2} \xi - \frac{1}{2} \eta \right) \cdot \left( \xi - \eta \right)        \\
		& \quad + \abs{\xi}^{r-2} \left( \frac{1}{2} \xi + \frac{1}{2} \eta \right) \cdot \left( \xi - \eta \right)
		+ \abs{\eta}^{r-2} \left( - \frac{1}{2} \xi - \frac{1}{2} \eta \right) \cdot \left( \xi - \eta \right)      \\
		& = \frac{1}{2} \left( \abs{\xi}^{r-2} + \abs{\eta}^{r-2} \right)
		\abs{ \xi - \eta }^2 
		+ \frac{1}{2} \left( \abs{\xi}^{r-2} - \abs{\eta}^{r-2} \right)
		\left( \abs{\xi}^2 - \abs{\eta}^2 \right) .
	\end{align*}
	Since $r \geq 2$, the second term is nonnegative and the inequality follows from the first term.
	
	For the case $1 < r < 2$, we note that the expressions are invariant under rotations and homogeneous with the same degree in both sides. Hence it is enough to consider the case $1=\abs{\xi} \geq \abs{\eta}$, $\xi = e_1$, $\eta= \eta_1 e_1 + \eta_2 e_2$. We split the argument in two cases. First, if $\eta_1 \leq 0$,
	\begin{align*}
		\left( \abs{\xi}^{r-1} - \abs{\eta}^{r-2} \eta_1 \right)
		\geq \abs{\xi}^{r-2} \left( \abs{\xi} - \eta_1 \right)
	\end{align*}
	and, if $0 \leq \eta_1 \leq \abs{\xi}$, by the mean value theorem,
	\begin{align*}
		\left( \abs{\xi}^{r-1} - \abs{\eta}^{r-2} \eta_1 \right)
		& \geq \left( \abs{\xi}^{r-1} - \eta_1^{r-1} \right) \geq (r-1) \abs{\xi}^{r-2} \left( \abs{\xi} - \eta_1 \right).
	\end{align*}
	Altogether, this yields
	\begin{align*}
		\left( \abs{\xi}^{r-2} \xi - \abs{\eta}^{r-2} \eta \right)
		\cdot \left( \xi - \eta \right)   
		& = \left( \abs{\xi}^{r-1} - \abs{\eta}^{r-2} \eta_1 \right) \left( \abs{\xi} - \eta_1 \right)
		+ \abs{\eta}^{r-2} \eta_2^2  \\
		& \geq  (r-1) \left( \abs{\xi} + \abs{\eta} \right) ^{r-2} \left( [\abs{\xi} - \eta_1]^2 + \eta_2^2 \right)   \\
		& = (r-1) \left( \abs{\xi} + \abs{\eta} \right) ^{r-2} \abs{\xi - \eta}^2.
	\end{align*}
\end{proof}

On the other hand, the (S$_+$)-property is a compactness-type condition of differential operators in the weak formulation. It commonly appears in the context of existence of solutions of elliptic partial differential equations by playing a role in the proof of the Cerami condition, the Palais-Smale condition, or other similar compactness-type condition of functionals, among other reasons. This is key to finding critical points of these functionals, and it is achieved via results like the mountain pass theorem, see for example the formulation in \cite[Theorem 5.4.6]{Papageorgiou-Radulescu-Repovs-2019a}. Alternatively, the (S$_+$)-property also participates in proving that differential operators in the weak formulation are a homeomorphism. Below, one can find the definition of the (S$_+$)-property together with two other relevant properties. 

\begin{definition}
	Let $X$ be a reflexive Banach space, $X^*$ its dual space and denote by $\langle \cdot \,, \cdot\rangle$ its duality pairing. Let $A\colon X\to X^*$. Then, 
	\begin{enumerate}
		\item[\textnormal{(i)}] $A$ is called monotone if $\langle A(u) - A(v), u - v \rangle \geq 0$, and strictly monotone if the inequality is strict for $u \neq v$;
		\item[\textnormal{(ii)}] $A$ is said to satisfy the $(S_+)$-property if $u_n \weak u$ in $X$ and $\limsup_{n\to \infty}$ $\langle A(u_n),u_n-u\rangle \leq 0$ imply $u_n\to u$ in $X$;
		\item[\textnormal{(iii)}] $A$ is called coercive if there exists some function $g\colon[0,\infty) \to \R$ such 
		that $\lim_{t \to + \infty} g(t) = + \infty$ and
		\begin{align*}
			\frac{\langle A(u), u \rangle}{\|u\|_X} \geq g(\| u \|_X ) \text{ for all } u \in X.
		\end{align*}
	\end{enumerate}
\end{definition}

The connection between monotonicity formulas and the (S$_+$)-property is that, for some popular differential operators, the former is used to prove the latter. In this section, we are going to see how this was done almost exclusively for operators with a power-law growth. Later, in Section \ref{Beyond_the_power-law_case}, we are going to mention a couple of recent works which have studied alternative growths and prove the same type of results for a more general class of operators.

In this text, let us denote by $\Omega \subseteq \R^N$ a bounded domain with Lipschitz boundary. For $1 \leq r \leq \infty$, we denote by $\Lp{r}$ the standard Lebesgue space equipped with the norm $\|\cdot\|_r$, and by $\Wp{r}$ and $\Wpzero{r}$ we denote the typical Sobolev spaces fitted with the norm $\|\cdot\|_{1,r}$ and $\|\cdot\|_{1,r,0}$, respectively.

Linear operators do not even require monotonicity inequalities since the operator has exactly the shape that yields the norm of the space. For example, the Laplacian operator
\begin{align*}
	A_{\Delta} \colon \Wpzero{2} \to \left[ \Wpzero{2} \right] ^*, 
	\qquad \langle A_{\Delta} (u) , v \rangle = \into \nabla u \cdot \nabla v \dx,
\end{align*}
satisfies 
\begin{align*}
	\langle A_{\Delta} (u) - A_{\Delta} (v) , u - v \rangle = \into (\nabla u - \nabla v)^2 \dx = \norm{ \nabla u - \nabla v }_2^2.
\end{align*}
Furthermore, this can be generalized to a more general linear operator as long as it is uniformly elliptic: let $M \colon \Omega \to \R^{N \times N}$ be a matrix regular enough and such that $\xi \cdot (M(x)\xi) \geq c \abs{\xi}^2$ for all $x \in \Omega$ and some constant $c>0$ independent of $x$. Then the linear operator
\begin{align*}
	A_{\text{Lin}} \colon \Wpzero{2} \to \left[ \Wpzero{2} \right] ^*, 
	\qquad \langle A_{\text{Lin}} (u) , v \rangle = \into \nabla u \cdot ( M(x) \nabla v ) \dx,
\end{align*}
satisfies 
\begin{align*}
	\langle A_{\text{Lin}} (u) - A_{\text{Lin}} (v) , u - v \rangle = \into (\nabla u - \nabla v) \cdot [M(x) (\nabla u - \nabla v)] \dx \geq c \norm{ \nabla u - \nabla v }_2^2.
\end{align*}

The $p$-Laplacian operator is the simplest example in which monotonicity formulas play a role in the proof of the (S$_+$)-property. However, one still can find alternative arguments, for example using that it is d-monotone (its definition can be found in \cite{Galewski-2021}) and the uniform convexity of $\Wpzero{p}$ as they do in \cite[Proposition 2]{Aizicovici-Papageorgiou-Staicu-2009}. One can prove the next theorem using both approaches.

\begin{theorem}[{See \cite[Proposition 2]{Aizicovici-Papageorgiou-Staicu-2009}, among others.}]
	Let $1 < p < \infty$, the operator $A_{\Delta_p} \colon \Wpzero{p} \to \left[ \Wpzero{p} \right] ^*$ be defined as
	\begin{align*}
		\langle A_{\Delta_p} (u) , v \rangle = \into \abs{\nabla u}^{p-2} \nabla u \cdot \nabla v \dx,
	\end{align*}
	and the functional $I_{\Delta_p} \colon \Wpzero{p} \to \R$ be defined as
	\begin{align*}
		I_{\Delta_p} (u) = \into \frac{ \abs{\nabla u}^p }{p} \dx.
	\end{align*}
	Then, $I_{\Delta_p}$ is $C^1$, $I_{\Delta_p}'=A_{\Delta_p}$, $A_{\Delta_p}$ satisfies the (S$_+$)-property, and it is strictly monotone, bounded, coercive, and a homeomorphism. 
\end{theorem}

Furthermore, the arguments used above for the $p$-Laplacian can be adapted to handle the case of the double phase operator with constant exponents, yielding the following theorem. The main difference is the need to use a Musielak-Orlicz space $\WHzero$ given by the function $\h(x,t) = t^p + \mu(x) t^q$. These spaces will be introduced with more detail in Section \ref{Musielak-Orlicz_spaces}.

\begin{theorem}[{See \cite[Proposition 3.1]{Liu-Dai-2018} and \cite[Proposition 2.1 and 2.2]{Perera-Squassina-2018}.}]
	Let $1 < p < q < p^*$, $0 \leq \mu \in \Lp{\infty}$, $\h(x,t) = t^p + \mu(x) t^q$, the operator $A_{\h} \colon \WHzero \to \left[ \WHzero \right] ^*$ be defined as
	\begin{align*}
		\langle A_{\h} (u) , v \rangle = \into \abs{\nabla u}^{p-2} \nabla u \cdot \nabla v + \mu(x) \abs{\nabla u}^{q-2} \nabla u \cdot \nabla v \dx,
	\end{align*}
	and the functional $I_{\h} \colon \WHzero \to \R$ be defined as
	\begin{align*}
		I_{\h} (u) = \into \frac{ \abs{\nabla u}^p }{p} + \mu(x) \frac{ \abs{\nabla u}^q }{q} \dx.
	\end{align*}
	Then, $I_{\h}$ is $C^1$, $I_{\h}'=A_{\h}$, $A_{\h}$ satisfies the (S$_+$)-property, and it is strictly monotone, bounded, coercive, and a homeomorphism. 
\end{theorem}

However, once we consider variable exponents, the strategy of using the d-monotonicity and the uniform convexity breaks down. One can see why in the proof of \cite[Proposition 2.2]{Perera-Squassina-2018}, since when one tries to use a H\"older inequality for more general spaces like $\Lp{p(\cdot)}$, a new sharp constant $2$ appears and destroys the argument. One could also try to use the modular of the space $\varrho_{p(\cdot)}$ instead of the norm, but no appropriate H\"older inequality is available in this shape. Alternatively, in the proof of \cite[Proposition 2]{Aizicovici-Papageorgiou-Staicu-2009}, it is referred to \cite[Proposition 4.4.1]{Gasinski-Papageorgiou-2005}, where assumption H(a)$_1$ (iv) fails. Hence, we need the monotonicity formulas to prove the (S$_+$)-property in this case. The following two theorems have been proven with this approach. In this case, let $\h(x,t) = t^{p(x)} + \mu(x) t^{q(x)}$ and consider its associated Musielak-Orlicz space $\WHzero$.

\begin{theorem}[{See \cite[Theorem 3.1]{Fan-Zhang-2003}.}]
	Let $p \in C(\close)$ with $1 < p(x) < \infty$ for all $x \in \close$, the operator $A_{\Delta_{p(\cdot)}} \colon \Wpzero{p(\cdot)} \to \left[ \Wpzero{p(\cdot)} \right] ^*$ be defined as
	\begin{align*}
		\langle A_{\Delta_{p(\cdot)}} (u) , v \rangle = \into \abs{\nabla u}^{p(x)-2} \nabla u \cdot \nabla v \dx,
	\end{align*}
	and the functional $I_{\Delta_{p(\cdot)}} \colon \Wpzero{p(\cdot)} \to \R$ be defined as
	\begin{align*}
		I_{\Delta_{p(\cdot)}} (u) = \into \frac{ \abs{\nabla u}^{p(x)} }{p(x)} \dx.
	\end{align*}
	Then, $I_{\Delta_{p(\cdot)}}$ is $C^1$, $I_{\Delta_{p(\cdot)}}'=A_{\Delta_{p(\cdot)}}$, $A_{\Delta_{p(\cdot)}}$ satisfies the (S$_+$)-property, and it is strictly monotone, bounded, coercive, and a homeomorphism. 
\end{theorem}

\begin{theorem}[{See \cite[Proposition 3.1 and Lemma 3.3]{Crespo-Blanco-Gasinski-Harjulehto-Winkert-2022}.}]
	Let $p,q \in C(\close)$ with $1 < p(x) < q(x) < p^*(x)$ for all $x \in \close$, $0 \leq \mu \in \Lp{\infty}$, $\h(x,t) = t^{p(x)} + \mu(x) t^{q(x)}$, the operator $A_{\h} \colon \WHzero \to \left[ \WHzero \right] ^*$ be defined as
	\begin{align*}
		\langle A_{\h} (u) , v \rangle = \into \abs{\nabla u}^{p(x)-2} \nabla u \cdot \nabla v + \mu(x) \abs{\nabla u}^{q(x)-2} \nabla u \cdot \nabla v \dx,
	\end{align*}
	and the functional $I_{\h} \colon \WHzero \to \R$ be defined as
	\begin{align*}
		I_{\h} (u) = \into \frac{ \abs{\nabla u}^{p(x)} }{p(x)} + \mu(x) \frac{ \abs{\nabla u}^{q(x)} }{q(x)} \dx.
	\end{align*}
	Then, $I_{\h}$ is $C^1$, $I_{\h}'=A_{\h}$, $A_{\h}$ satisfies the (S$_+$)-property, and it is strictly monotone, bounded, coercive, and a homeomorphism. 
\end{theorem}

One final note is that these kind of results can also be proven for operators involving the values of the solution and not only the gradients. If the values of the solution play a similar role as the gradients, the proof can be made mostly analogously. They can even have different exponents or vanishing weights as long as some compatibility conditions are met. One example of such results can be found in \cite[Proposition 3.3]{Amoroso-Crespo-Blanco-Pucci-Winkert-2023}. If the values of the solution play a different role, for example they are multiplied by the gradient, then the situation can be more complicated. For a statement for a general Leray-Lions operator, check \cite[Theorem 2.109]{Carl-Le-Montreanu-2007}.

As we have seen, all these operators have power-law growth, the (S$_+$)-property can be proven using the usual monotonicity formula and, in the case of variable exponents, the alternative argument fails. In light of this observation, one could think whether a similar strategy can be applied to a more general class of operators, which may not have power-law growth. Of course, this would require substituting the usual monotonicity formulas with another adequate tool. In Section \ref{Beyond_the_power-law_case}, such results are proven and the necessary tools are presented. As we will deal with general growths that may not be power-laws, we need to first introduce Musielak-Orlicz spaces in Section \ref{Musielak-Orlicz_spaces}.

\section{Musielak-Orlicz spaces}\label{Musielak-Orlicz_spaces}

In Section \ref{Beyond_the_power-law_case} we are going to consider operators which do not have power-law growth. The most immediate consequence is that the usual Lebesgue and Sobolev spaces are not adequate for the study of such operators, and hence we will use Musielak-Orlicz spaces. Let us introduce them first. For this purpose, we follow the definitions and results from the book of Harjulehto-H\"ast\"o \cite{Harjulehto-Hasto-2019}. For the rest of this section let us denote by $(A,\Sigma,\mu)$ a $\sigma$-finite, complete measure space with $\mu \not \equiv 0$, while $\Omega$ still denotes a bounded domain in $\R^N$ with $N \geq 2$ and Lipschitz boundary $\partial \Omega$.

\begin{definition}
	Let  $\ph \colon A \times (0,+\infty) \to \R$. We say that
	\begin{enumerate}
		\item[\textnormal{(i)}]
		$\ph$ is almost increasing in the second variable if there exists $a \geq 1$ such that $\ph(x,s) \leq a \ph(s,t)$ for all $0 < s < t$ and for a.a.\,$x \in A$;
		\item[\textnormal{(ii)}]
		$\ph$ is almost decreasing in the second variable if there exists $a \geq 1$ such that $a \ph(x,s) \geq \ph(x,t)$ for all $0 < s < t$ and for a.a.\,$x \in A$.
	\end{enumerate}
	
	Let $\ph \colon A \times (0,+\infty) \to \R$ and $p,q>0$. We say that $\ph$ satisfies the property
	\begin{enumerate}[leftmargin=2cm]
		\item[\textnormal{(Inc)}$_p$]
		if $t^{-p}\ph(x,t)$ is increasing in the second variable;
		\item[\textnormal{(aInc)}$_p$]
		if $t^{-p}\ph(x,t)$ is almost increasing in the second variable;
		\item[\textnormal{(Dec)}$_q$]
		if $t^{-q}\ph(x,t)$ is decreasing in the second variable;
		\item[\textnormal{(aDec)}$_q$]
		if $t^{-q}\ph(x,t)$ is almost decreasing in the second variable.
	\end{enumerate}
	
	Without subindex, that is \textnormal{(Inc)}, \textnormal{(aInc)}, \textnormal{(Dec)}, and \textnormal{(aDec)}, it indicates that there exists some $p>1$ or $q<\infty$ such that the condition holds.
\end{definition}

\begin{definition}
	A function $\ph \colon A \times [0,+\infty) \to [0,+\infty]$ is said to be a generalized $\Phi$-function if $\ph$ is measurable in the first variable, increasing in the second variable, and satisfies $\ph(x,0)=0$, $\lim_{t\to 0^+} \ph(x,t) = 0$, and $\lim_{t \to +\infty} \ph(x,t) = +\infty$ for a.a.\,$x \in A$. Moreover, we say that
	\begin{enumerate}
		\item[\textnormal{(i)}]
		$\ph$ is a generalized weak $\Phi$-function if it satisfies \textnormal{(aInc)}$_1$ on $A \times (0,+\infty)$;
		\item[\textnormal{(ii)}]
		$\ph$ is a generalized convex $\Phi$-function if $\ph(x,\cdot)$ is left-continuous and convex for a.a.\,$x \in A$;
		\item[\textnormal{(iii)}]
		$\ph$ is a generalized strong $\Phi$-function if $\ph(x,\cdot)$ is continuous in the topology of $[0,\infty]$ and convex for a.a.\,$x \in A$.
	\end{enumerate}
	
	If $\ph$ does not depend on $x$, i.e. $\ph(x,t)=\ph(x_0,t)$ for $x_0 \in \Omega$ and all $x \in \Omega$, then we say it is a $\Phi$-function.
\end{definition}

\begin{definition}
	Let $\ph \colon A \times [0,+\infty) \to [0,+\infty]$. We denote by $\ph^*$ the conjugate function of $\ph$, which is defined for $x \in A$ and $s \geq 0$ by
	\begin{align*}
		\ph^*(x,s) = \sup_{t \geq 0} (ts - \ph(x,t)).
	\end{align*}
	
	Furthermore, we say that 
	\begin{enumerate}
		\item[\textnormal{(i)}]
		$\ph$ is doubling (or satisfies the $\Delta_2$-condition) if there exists a constant $K \geq 2$ such that
		\begin{align*}
			\ph(x,2t) \leq K \ph(x,t)
		\end{align*}
		for all $t \in (0,+\infty]$ and for a.a.\,$x \in A$;
		\item[\textnormal{(ii)}]
		$\ph$ satisfies the $\nabla_2$-condition if $\ph^*$ satisfies the $\Delta_2$ condition.
	\end{enumerate}
\end{definition}

\begin{lemma}[{\cite[Lemma 2.2.6 and Corollary 2.4.11]{Harjulehto-Hasto-2019}}]
	\label{Le:equivalences}
	Let $\ph \colon A \times [0,+\infty) \to [0,+\infty]$ be a generalized weak $\Phi$-function. Then,
	\begin{enumerate}
		\item[\textnormal{(i)}]
		it satisfies the $\Delta_2$-condition if and only if it satisfies \textnormal{(aDec)};
		\item[\textnormal{(ii)}]
		if it is a generalized convex $\Phi$-function, it satisfies the $\Delta_2$ condition if and only if it satisfies \textnormal{(Dec)};
		\item[\textnormal{(iii)}]
		it satisfies the $\nabla_2$-condition if and only if it satisfies \textnormal{(aInc)}.
	\end{enumerate}
\end{lemma}

\begin{remark}[{Paragraph after \cite[Lemma 2.2.6 and Definition 2.5.22]{Harjulehto-Hasto-2019}}]
	\label{Re:aInc-Inc-equivalence}
	One important consequence of Lemma \ref{Le:equivalences} is that a generalized convex $\Phi$-function
	\begin{enumerate}
		\item[\textnormal{(i)}] satisfies \textnormal{(aInc)} if and only if it satisfies \textnormal{(Inc)} (possibly with the exponent of \textnormal{(Inc)} smaller than the one of \textnormal{(aInc)});
		\item[\textnormal{(ii)}] satisfies \textnormal{(aDec)} if and only if it satisfies \textnormal{(Dec)}
		(possibly with the exponent of \textnormal{(Dec)} bigger than the one of \textnormal{(aDec)}).
	\end{enumerate}
\end{remark}

For some inequalities that will be used later, we need the following definition from \cite{Diening-Harjulehto-Hasto-Ruzicka-2011}. Note that it is a stronger concept than the generalized strong $\Phi$-function.

\begin{definition}
	Let $\ph \colon [0,+\infty) \to [0,+\infty)$. We say that it is an N-function if it is a convex, strictly increasing function such that $\lim_{t \to 0^+} \ph(t)/t = 0$ and $\lim_{t \to \infty} \ph(t)/t = \infty$.
	
	Let $\ph \colon A \times [0,\infty) \to [0,\infty)$. We say that it is a generalized N-function if $\ph(\cdot,t)$ is measurable for all $t \in [0,\infty)$ and $\ph(x,\cdot)$ is an N-function for a.a. $x \in A$.
\end{definition}

\begin{lemma}[{\cite[Remark 2.6.7]{Diening-Harjulehto-Hasto-Ruzicka-2011}}]
	\label{Le:EquivalenceFirstDerivative}
	Let $\ph \colon A \times [0,+\infty) \to [0,+\infty)$ be a generalized N-function such that it satisfies \textnormal{(aDec)}. Then, there exists $C_1,C_2 > 0$ independent of $x \in A$ such that for a.a.\,$x \in A$ and all $t \in [0,\infty)$,
	\begin{align*}
		C_1 \ph(x,t) \leq \ph'(x,t) t \leq C_2 \ph(x,t).
	\end{align*}
\end{lemma}

\begin{lemma}[{\cite[Lemma 2.6.11]{Diening-Harjulehto-Hasto-Ruzicka-2011}}]
	\label{Le:ConjugateInequality}
	Let $\ph \colon [0,+\infty) \to [0,+\infty)$ be an N-function. Then, for all $t \in [0,\infty)$, it holds 
	\begin{align*}
		\ph^*( \ph'(t) ) \leq t \ph'(t). 
	\end{align*}
\end{lemma}

\begin{proposition}[{\cite[Lemma 3.1.3, Lemma 3.2.2, Lemma 3.3.5, Theorem 3.3.7, Theorem 3.5.2 and Theorem 3.6.6]{Harjulehto-Hasto-2019}}]
	\label{Prop:AbstractBanach}
	Let $\ph \colon A \times [0,+\infty) \to [0,+\infty]$ be a generalized weak $\Phi$-function and let its associated modular be
	\begin{align*}
		\varrho_\ph (u) = \int_A \ph(x,\abs{u(x)}) \diff \mu (x).
	\end{align*}
	Then, the set
	\begin{align*}
		L^\ph (A) = \{ u \colon A \to \R \text{ measurable}\,:\, \varrho_\ph (\lambda u) < \infty \text{ for some } \lambda > 0 \}
	\end{align*}
	equipped with the associated Luxemburg quasi-norm (i.e. a norm, except it only satisfies a weaker version of the triangle inequality involving a multiplicative constant $K \geq 1$)
	\begin{align*}
		\norm{u}_\ph = \inf \left\lbrace \lambda > 0 \,:\, \varrho_\ph \left( \frac{u}{\lambda} \right)  \leq 1 \right\rbrace
	\end{align*}
	is a quasi Banach space (i.e. it is topologically complete in the equipped quasi-norm). Furthermore, if $\mu(A) < \infty$, every sequence convergent in $\norm{\cdot}_\ph$ is also convergent in measure; if $\ph$ is a generalized convex $\Phi$-function, $\norm{\cdot}_\ph$ is a norm, so $L^\ph (A)$ is a Banach space; if $\ph$ satisfies \textnormal{(aDec)}, it holds that
	\begin{align*}
		L^\ph (A) = \{ u \colon A \to \R \text{ measurable} \,:\, \varrho_\ph (u) < \infty \};
	\end{align*}
	if $\ph$ satisfies \textnormal{(aDec)} and $\mu$ is separable, then $L^\ph (A)$ is separable; and if $\ph$ satisfies \textnormal{(aInc)} and \textnormal{(aDec)}, $L^\ph (A)$ possesses an equivalent, uniformly convex norm, hence it is reflexive.
\end{proposition}

\begin{proposition}[{\cite[Lemma 3.2.9]{Harjulehto-Hasto-2019}}]
	\label{Prop:AbstractNormModular}
	Let $\ph \colon A \times [0,+\infty) \to [0,+\infty]$ be a generalized weak $\Phi$-function that satisfies \textnormal{(aInc)}$_p$ and \textnormal{(aDec)}$_q$, with $1 \leq p \leq q < \infty$. Then,
	\begin{align*}
		\frac{1}{a} \min \left\lbrace \norm{u}_\ph^p ,  \norm{u}_\ph^q \right\rbrace
		\leq \varrho_\ph (u)
		\leq a \max \left\lbrace \norm{u}_\ph^p ,  \norm{u}_\ph^q \right\rbrace
	\end{align*}
	for all measurable functions $u \colon A \to \R$, where $a$ is the maximum of the constants of \textnormal{(aInc)}$_p$ and \textnormal{(aDec)}$_q$.
\end{proposition}

\begin{proposition}[{\cite[Theorem 3.2.6]{Harjulehto-Hasto-2019}}]
	Let $\ph, \psi \colon A \times [0,+\infty) \to [0,+\infty]$ be generalized weak $\Phi$-functions and let $\mu$ be atomless. Then, $L^\ph (A) \hookrightarrow L^\psi (A)$ if and only if there exits $K>0$ and $h \in L^1 (A)$ with $\norm{h}_1 \leq 1$ such that, for all $t \geq 0$ and for a.a.\,$x \in \Omega$,
	\begin{align*}
		\psi\left(  x,\frac{t}{K} \right) \leq \ph (x,t) + h(x).
	\end{align*}
\end{proposition}

\begin{proposition}[{\cite[Lemma 3.2.11]{Harjulehto-Hasto-2019}}]
	\label{Prop:AbstractHoelder}
	Let $\ph \colon A \times [0,+\infty) \to [0,+\infty]$ be a generalized weak $\Phi$-function. Then,
	\begin{align*}
		\int_A \abs{u} \abs{v} \diff \mu (x) \leq 2 \norm{u}_{\ph} \norm{v}_{\ph^*} \quad \text{for all } u \in L^{\ph}(A), v \in L^{\ph^*}(A).
	\end{align*}
	Moreover, the constant $2$ is sharp.
\end{proposition}

\begin{proposition}[{\cite[Theorem 6.1.4 and Theorem 6.1.9]{Harjulehto-Hasto-2019}}]
	Let $\ph \colon \Omega \times [0,+\infty) \to [0,+\infty]$ be a generalized weak $\Phi$-function such that $\Lp{\ph} \subseteq L^1_{\loc} (\Omega)$ and $k\geq 1$. Then, the set
	\begin{align*}
		W^{k,\ph} (\Omega) = \{ u \in \Lp{\ph} \,:\, \partial_\alpha u \in \Lp{\ph} \text{ for all } \abs{\alpha} \leq k \},
	\end{align*}
	where we consider the modular
	\begin{align*}
		\varrho_{k,\ph} (u) = \sum_{0 \leq \abs{\alpha} \leq k } \varrho_\ph(\partial_\alpha u)
	\end{align*}
	and the associated Luxemburg quasi-norm (i.e. a norm, except it only satisfies a weaker version of the triangle inequality involving a multiplicative constant $K \geq 1$)
	\begin{align*}
		\norm{u}_{k,\ph} = \inf \left\lbrace \lambda > 0 : \varrho_{k,\ph} \left( \frac{u}{\lambda} \right)  \leq 1 \right\rbrace
	\end{align*}
	is a quasi Banach space (i.e. it is topologically complete in the equipped quasi-norm). Analogously, the set
	\begin{align*}
		W^{k,\ph}_0 (\Omega) = \overline{C_0^\infty (\Omega)}^{\norm{\cdot}_{k,\ph}},
	\end{align*}
	where $C_0^\infty (\Omega)$ are the functions in $C^\infty (\Omega)$ with compact support, equipped with the same modular and quasi-norm, is also a quasi Banach space.
	
	Furthermore, if $\ph$ is a generalized convex $\Phi$-function, $\norm{\cdot}_{k,\ph}$ is a norm, so both spaces are Banach spaces; if $\ph$ satisfies \textnormal{(aDec)}, then both spaces are separable; and if $\ph$ satisfies \textnormal{(aInc)} and \textnormal{(aDec)}, both spaces possess an equivalent, uniformly convex norm, hence they are reflexive.
\end{proposition}

\begin{proposition}[{\cite[Lemma 3.2.9]{Harjulehto-Hasto-2019}}]
	\label{Prop:AbstractoneNormModular}
	Let $\ph \colon \Omega \times [0,+\infty) \to [0,+\infty]$ be a generalized weak $\Phi$-function that satisfies \textnormal{(aInc)}$_p$ and \textnormal{(aDec)}$_q$, with $1 \leq p \leq q < \infty$. Then, 
	\begin{align*}
		\frac{1}{a} \min \left\lbrace \norm{u}_{k,\ph}^p ,  \norm{u}_{k,\ph}^q \right\rbrace
		\leq \varrho_{k,\ph} (u)
		\leq a \max \left\lbrace \norm{u}_{k,\ph}^p ,  \norm{u}_{k,\ph}^q \right\rbrace
	\end{align*}
	for all $u \in W^{k,\ph} (\Omega)$, where $a$ is the maximum of the constants of \textnormal{(aInc)}$_p$ and \textnormal{(aDec)}$_q$.
\end{proposition}

\section{Beyond the power-law case}\label{Beyond_the_power-law_case}

As stated in Section \ref{Power-law_case}, here we are going to deal with operators which do not exhibit a power-law growth, but still can be proven to have the same kind of properties, with special emphasis on the (S$_+$)-property. There exist some previous works which already deal with operators whose growth is governed by a $\Phi$-function, see \cite{Barletta-2020,Barletta-2024,Barletta-Cianchi-2017} and the references therein. However, although some of these works handle generalized $\Phi$-functions, their assumptions always depend on some $\Phi$-function independent of $x$. In this work, we obtain results without assumptions uniform in $x$. 

Operators with a growth of a power-law times a logarithm have been recently studied in \cite[Theorem 4.1 and Theorem 4.4]{Arora-Crespo-Blanco-Winkert-2023} and \cite[Lemma 4.2]{Tran-Nguyen-2023}. In the first reference, monotonicity inequalities like those of Lemma \ref{Le:MonotoneInequality} are modified to further have an increasing function next to the power-law, and then these are used to prove the result, while in the second reference they use a previous, more general formula which involves the second derivative of the corresponding $\Phi$-function, where all terms but the useful one turn out nonnegative and hence they can be ignored. In particular, the following result holds.

\begin{theorem}[{See \cite[Theorem 4.1 and 4.4]{Arora-Crespo-Blanco-Winkert-2023}.}]
	Let $p,q \in C(\close)$ with $1 < p(x) \leq q(x) < p^*(x)$ for all $x \in \close$, $0 \leq \mu \in \Lp{\infty}$, $\hlog(x,t) = t^{p(x)} + \mu(x) t^{q(x)}\log(e+t)$, the operator $A_{\hlog} \colon \WHlogzero \to \left[ \WHlogzero \right] ^*$ be defined as
	\begin{align*}
		\langle A_{\hlog} (u) , v \rangle &= \into \abs{\nabla u}^{p(x)-2} \nabla u \cdot \nabla v \\
		& \quad + \mu(x) \abs{\nabla u}^{q(x)-2}\left[ \log(e + \abs{\nabla u}) + \frac{\abs{\nabla u}}{q(x) (e + \abs{\nabla u})} \right]  \nabla u \cdot \nabla v \dx,
	\end{align*}
	and the functional $I_{\hlog} \colon \WHlogzero \to \R$ be defined as
	\begin{align*}
		I_{\hlog} (u) = \into \frac{ \abs{\nabla u}^{p(x)} }{p(x)} + \mu(x) \frac{ \abs{\nabla u}^{q(x)} }{q(x)}\log(e + \abs{\nabla u}) \dx.
	\end{align*}
	Then, $I_{\hlog}$ is $C^1$, $I_{\hlog}'=A_{\hlog}$, $A_{\hlog}$ satisfies the (S$_+$)-property, and it is strictly monotone, bounded, coercive, and a homeomorphism. 
\end{theorem}

In this context it is also important to mention the work \cite{Diening-Ettwein-2008}. The objective of this paper is different, there is no attempt to prove results like the theorems of Section \ref{Power-law_case} involving the (S$_+$)-property and the other claims. However, key inequalities are significantly generalized in this work, such as, among others, the monotonicity formula that will be used in this section. Its statement can be found below. 

\begin{definition}
	Let $\ph$ be a generalized convex $\Phi$-function. We define
	$a_\ph \colon \Omega \times \R^N \to \R^N$ as 
	\begin{align*}
		a_\ph (x,\xi) &= \nabla_\xi \ph(x,\abs{\xi}) = \ph'(x,\abs{\xi}) \frac{\xi}{\abs{\xi}}.
	\end{align*}
\end{definition}

\begin{lemma}[{See \cite[Lemma 21]{Diening-Ettwein-2008}.}]
	\label{Le:GeneralizedMono}
	Let $\ph$ be an N-function such that it satisfies \textnormal{(aInc)}, \textnormal{(aDec)}, $\ph \in C^2(0,\infty)$, and there exist $C_1,C_2 >0$ such that for all $t \in (0,\infty)$,
	\begin{equation*}
		C_1 \ph'(t) \leq t \ph''(t) \leq C_2 \ph'(t).
	\end{equation*}
	Then, there exists $C>0$ such that, for all $\xi,\eta \in \R^N$,
	\begin{align*}
		\langle a_\ph(\xi) -  a_\ph(\eta) , \xi - \eta \rangle 
		\geq C \ph''( \abs{\xi} + \abs{\eta} ) \abs{\xi - \eta}^2.
	\end{align*}
\end{lemma}

\begin{remark}
	\label{Re:GeneralizedMono}
	Note that this is a result only for N-functions and not generalized N-functions. This means that, given a generalized N-function $\ph$ with good enough properties, we could then apply it to each $\ph(x,\cdot)$ with $x \in \Omega$, but then the constant $C$ would depend on $x$, that is 
	\begin{align*}
		( a_\ph(x,\xi) -  a_\ph(x,\eta) ) \cdot ( \xi - \eta ) 
		\geq C_x \ph''(x, \abs{\xi} + \abs{\eta} ) \abs{\xi - \eta}^2.
	\end{align*}
	
	However, for our purposes, the fact that $C$ depends on $x$ will not be an obstacle, as it will be seen later. Another thing to note is that we do not even know if $C_x$ is measurable, so it must not appear in any integrand.
\end{remark}

The result above motivates the following assumption.

\begin{definition}
	Let $\ph$ be a generalized N-function. We say that it fulfills \textnormal{(Mono1)} if it satisfies \textnormal{(aInc)}, \textnormal{(aDec)}, $\ph(x,\cdot) \in C^2(0,\infty)$ for a.a. $x \in \Omega$, and there exist $C_{1,x},C_{2,x} >0$ such that, for all $t \in (0,\infty)$ and a.a. $x \in \Omega$,
	\begin{equation*}
		C_{1,x} \ph'(x,t) \leq t \ph''(x,t) \leq C_{2,x} \ph'(x,t).
	\end{equation*}
\end{definition}

Based on the inequality above, we prove novel, more general results in the style of those mentioned in Section \ref{Power-law_case} and Section \ref{Beyond_the_power-law_case}. Let us introduce now the operator that will be studied in this section. 

\begin{definition}
	Let $\ph$ be a generalized convex $\Phi$-function. We define $A_\ph \colon \Wphzero$ $\to \left[ \Wphzero \right] ^*$ as 
	\begin{align*}
		\langle A_{\ph} (u) , v \rangle &= \into a_\ph (x,\nabla u) \cdot \nabla v \dx 
		= \into \ph'(x,\abs{ \nabla u }) \frac{ \nabla u }{\abs{ \nabla u }} \cdot \nabla v \dx,
	\end{align*}
	and the functional $I_{\ph} \colon \Wphzero \to \R$ as
	\begin{align*}
		I_{\ph} (u) = \into \ph(x,\abs{ \nabla u }) \dx.
	\end{align*}
\end{definition}

First, the differentiability of $I_{\ph}$ is proved under strong enough assumptions.

\begin{proposition} \label{Prop:Differentiabilityph}
	Let $\ph$ be a generalized N-function such that it satisfies \textnormal{(aInc)}, \textnormal{(aDec)}, and $\ph(x,\cdot) \in C^1(0,\infty)$ for a.a.\,$x \in \Omega$. Then, $I_{\ph}$ is $C^1$ and $I_{\ph}' =  A_{\ph}$.
\end{proposition}

\begin{proof}
	As is usual in this kind of results, the proof is divided in two parts: the Gateaux differentiability and the continuity of the derivative.
	
	We begin with the Gateaux differentiability. Let $u,v \in \Wphzero$ and $t \in \R$. Considering the points where $\nabla u \neq 0$, the mean value theorem yields $\theta_{x,t} \in (0,1)$ such that
	\begin{align*}
		\frac{\ph(x,\abs{\nabla u + t \nabla h}) - \ph(x,\abs{\nabla u})}{t} 
		& = \ph'(x,\abs{\nabla u + t \theta_{x,t} \nabla h}) \frac{\nabla u + t \theta_{x,t} \nabla h}{\abs{\nabla u + t \theta_{x,t} \nabla h}} \cdot \nabla h \\
		& \xrightarrow{t \to 0} \ph'(x,\abs{\nabla u }) \frac{\nabla u}{\abs{\nabla u }} \cdot \nabla h.
	\end{align*}
	On the other hand, note that by Remark \ref{Re:aInc-Inc-equivalence}, $\ph$ satisfies \textnormal{(Inc)}, which implies $\ph'(x,0)=0$ for a.a.\,$x \in \Omega$. For this reason, the limit above also holds in the points where $\nabla u = 0$. Next, for $t \leq 1$, since $\theta_{x,t} \in (0,1)$ and by Lemma \ref{Le:EquivalenceFirstDerivative}, we have that
	\begin{align*}
		\ph'(x,\abs{\nabla u + t \theta_{x,t} \nabla h}) \frac{\nabla u + t \theta_{x,t} \nabla h}{\abs{\nabla u + t \theta_{x,t} \nabla h}} \cdot \nabla h 
		& \leq \ph'(x,\abs{\nabla u} + \abs{\nabla h}) (\abs{\nabla u} + \abs{\nabla h}) \\
		& \leq C \ph(x,\abs{\nabla u} + \abs{\nabla h}) \\
		& \leq C \ph(x,2 \max \{\abs{\nabla u} , \abs{\nabla h}\}) \\
		& \leq C K \ph(x,\max \{\abs{\nabla u} , \abs{\nabla h}\}) \\
		& \leq C K \left[ \ph(x, \abs{\nabla u} ) + \ph(x, \abs{\nabla h} ) \right] ,
	\end{align*}
	which is an $L^1$-majorant uniform in $t$. Hence, by the dominated convergence theorem, the Gateaux derivative exists and is given by $A_\ph$.
	
	We prove now that $I_\ph$ is $C^1$. Let $u_n \to u$ in $\Wphzero$ and $v \in \Wphzero$ with $\normoneph{v} \leq 1$. By the H\"older's inequality, see Proposition \ref{Prop:AbstractHoelder},
	\begin{align*}
		\abs{ \langle A_{\ph} (u_n) - A_{\ph} (u) , v \rangle} 
		& \leq \into \abs{a_{\ph} (\nabla u_n) - a_{\ph} (\nabla u)} \abs{\nabla v} \dx \\
		& \leq 2 \norm{a_{\ph} (\nabla u_n) - a_{\ph} (\nabla u)}_{\ph^*}.
	\end{align*}
	In order to see that this norm converges to zero, we prove that the modular of the same function converges to zero via Vitali's theorem. By the strong convergence in $\Wphzero$ and Proposition \ref{Prop:AbstractBanach}, $\nabla u_n \to \nabla u$ in measure. By straightforward computations (pass to an a.e. convergent subsequence and apply the subsequence principle), we have that $a_{\ph} (\nabla u_n) \to a_{\ph} (\nabla u)$ in measure. For the uniform integrability, we repeat the argument with the $\Delta_2$-condition used above, and from Lemma \ref{Le:ConjugateInequality} and Lemma \ref{Le:EquivalenceFirstDerivative}, we get for a.a.\,$x \in \Omega$,
	\begin{align*}
		\ph^*(x, \abs{a_{\ph} (\nabla u_n) - a_{\ph} (\nabla u)} )
		& \leq \ph^*(x, \ph' (x,\abs{\nabla u_n}) + \ph' (x,\abs{\nabla u}) ) \\
		& \leq K \left[  \ph^*(x, \ph' (x,\abs{\nabla u_n})) + \ph^*(x,\ph' (x,\abs{\nabla u}) ) \right] \\
		& \leq K \left[  \abs{\nabla u_n} \ph' (x,\abs{\nabla u_n}) + \abs{\nabla u} \ph' (x,\abs{\nabla u})  \right] \\
		& \leq K C \left[  \ph (x,\abs{\nabla u_n}) + \ph (x,\abs{\nabla u})  \right]. 
	\end{align*}
	Since $\modph{\nabla u_n - \nabla u} \to 0$, the previous expression is uniformly integrable, and by Vitali's theorem we conclude $\varrho_{\ph^*} (a_{\ph} (\nabla u_n) - a_{\ph} (\nabla u)) \to 0$, or equivalently $\norm{a_{\ph} (\nabla u_n) - a_{\ph} (\nabla u)}_{\ph^*} \to 0$.
\end{proof}

Next, we prove monotonicity and compactness-type conditions of the operator $A_\ph$. For this purpose, we need a new lemma and two extra assumptions. Since in $A_\ph$ we are working only with the gradients in the operator, possibly in the context of a boundary value problem with homogeneous Dirichlet boundary conditions, one needs a Poincar\'e inequality to obtain the right norm in the space. It is possible to give general conditions on $\ph$ to ensure that this kind of inequality holds, see for example \cite[Theorem 6.2.8]{Harjulehto-Hasto-2019} or \cite[Theorem 1.1 and 1.2]{Fan-2012}. However, the assumptions of these results are by no means sharp, and depending on the specific operator one can achieve such an inequality with weaker assumptions, even for operators with complicated shapes, see for example \cite[Proposition 2.18]{Crespo-Blanco-Gasinski-Harjulehto-Winkert-2022} or \cite[Proposition 3.9]{Arora-Crespo-Blanco-Winkert-2023}. It is for this reason that it makes more sense for our results to assume Poincar\'e inequality and leave its proof for each specific considered function. The other assumption plays a role in the (S$_+$)-property.

\begin{definition}
	Let $\ph$ be a generalized weak $\Phi$-function. We say that it satisfies \textnormal{(Poin)} if $\Wph \hookrightarrow \Lph$ compactly. 
\end{definition}

\begin{remark}
	The assumption \textnormal{(Poin)} implies that there exists $C>0$ such that, for all $u \in \Wphzero$,
	\begin{align*}
		\normph{u} \leq C \normph{\nabla u}.
	\end{align*}
	This can be proven with standard arguments (see, for example, \cite[Proposition 2.18]{Crespo-Blanco-Gasinski-Harjulehto-Winkert-2022}). Furthermore, this implies that we can equip $\Wphzero$ with a norm equivalent to $\normoneph{\cdot}$ given by
	\begin{align*}
		\normonephzero{u} = \normph{ \nabla u}.
	\end{align*}
\end{remark}

\begin{definition}
	Let $\ph$ be a generalized convex $\Phi$-function such that $\ph(x,\cdot) \in C^2(0,\infty)$ for a.a.\,$x \in \Omega$. We say it satisfies \textnormal{(Mono2)} if for a.a.\,$x \in \Omega$, either $\ph''(x,\cdot)$ is almost increasing, or $\ph''(x,\cdot)$ is almost decreasing and for any $c>0$, $\lim_{t \to \infty} \ph''(x,c+t) (c-t)^2 = \infty$. The constant for almost increasing/decreasing $a_x$ depends on $x \in \Omega$.
\end{definition}

\begin{lemma}[Young's inequality for generalized N-functions.]
	\label{Le:YoungIneqph}
	Let $\ph$ be a generalized N-function. Then, for all $t \in [0,\infty)$ and a.a.\,$x \in \Omega$, it holds that
	\begin{align*}
		s \ph'(x,t) \leq \ph(x,s) + t \ph'(x,t) - \ph(x,t).
	\end{align*}
\end{lemma}

\begin{proof}
	Let $h \colon [0,\infty) \to [0,\infty)$ be continuous and strictly increasing, and with $h(0)=0$. By the classical result \cite[Theorem 156]{Hardy-Littlewood-Polya-1934} we know that
	\begin{align*}
		s h(t) \leq \int_0^s h(r) \dr + \int_0^{h(t)} h^{-1} (r) \dr,
	\end{align*}
	and by another classical result we can reshape the integral of the inverse as
	\begin{align*}
		\int_0^{h(t)} h^{-1} (r) \dr = t h(t) - \int_0^t h(r) \dr,
	\end{align*}
	which taken together yields
	\begin{align*}
		s h(t) \leq \int_0^s h(r) \dr + t h(t) - \int_0^t h(r) \dr.
	\end{align*}
	For each fixed $x \in \Omega$, inserting $h(t) = \ph'(x,t)$, one obtains the desired result.
\end{proof}

\begin{proposition}
	\label{Prop:MonoAndSplus}
	Let $\ph$ be a generalized N-function that satisfies \textnormal{(Mono1)}. Then the operator $A_\ph$ is monotone. Furthermore, if $\ph$ satisfies $\Lph \hookrightarrow \Lp{1}$, then the operator $A_\ph$ is strictly monotone, and if $\ph$ satisfies \textnormal{(Poin)} and \textnormal{(Mono2)}, then $A_\ph$ is strictly monotone and fulfills the (S$_+$)-property.
\end{proposition}

\begin{proof}
	For the monotonicity, consider $u,v \in \Wphzero$. Lemma \ref{Le:GeneralizedMono} and Remark \ref{Re:GeneralizedMono} yield that a.e. in $\Omega$,
	\begin{align*}
		& \quad ( a_\ph(x,\nabla u) -  a_\ph(x,\nabla v) ) \cdot ( \nabla u - \nabla v ) \\
		& \geq C_x \ph''(x, \abs{\nabla u} + \abs{\nabla v} ) \abs{\nabla u - \nabla v}^2  \geq 0.
	\end{align*}
	Therefore, $\langle A_\ph(u)- A_\ph(v), u - v \rangle \geq 0$. Furthermore, if $\langle A_\ph(u)- A_\ph(v), u - v \rangle = 0$, we know that for a.a. $x \in \Omega$,
	\begin{align*}
		0 & = ( a_\ph(x,\nabla u) -  a_\ph(x,\nabla v) ) \cdot ( \nabla u - \nabla v ) \\
		& \geq C_x \ph''(x, \abs{\nabla u} + \abs{\nabla v} ) \abs{\nabla u - \nabla v}^2 ,
	\end{align*}
	hence $\nabla u = \nabla v$ a.e. in $\Omega$. If $\Lph \hookrightarrow \Lp{1}$, then $\Wphzero \hookrightarrow \Wpzero{1}$, so $u = v$ a.e. in $\Omega$. Alternatively, if \textnormal{(Poin)} is satisfied, it also follows that $u = v$ a.e. in $\Omega$.
	
	For the (S$_+$)-property, consider a sequence $u_n \weak u $ in $\Wphzero$ which also satisfies $\limsup_{n\to \infty} \langle A_\ph (u_n),u_n-u\rangle \leq 0$. By the strict monotonicity and weak convergence, we can deduce that
	\begin{equation*}
		\lim_{n\to \infty} \langle A_\ph (u_n) - A_\ph (u) ,u_n-u\rangle = 0.
	\end{equation*}
	
	\noindent {\bf Claim:} $\nabla u_n \to \nabla u$ in measure. 
	
	From the limit above, we derive that there exists a subsequence $u_{n_k}$ such that
	\begin{align*}
		( a_\ph(\cdot,\nabla u) -  a_\ph(\cdot,\nabla u_{n_k}) ) \cdot ( \nabla u - \nabla u_{n_k} ) \to 0 \qquad \text{a.e. in } \Omega.
	\end{align*}
	Using Lemma \ref{Le:GeneralizedMono} and Remark \ref{Re:GeneralizedMono} yields
	\begin{align*}
		\ph''(\cdot, \abs{\nabla u} + \abs{\nabla u_{n_k}} ) \abs{\nabla u - \nabla u_{n_k}}^2 \to 0 \qquad \text{a.e. in } \Omega.
	\end{align*}
	We divide the proof in two cases now. In the subset of $\Omega$ where $\ph''(x,\cdot)$ is almost increasing, by assumption \textnormal{(Mono1)} and Lemma \ref{Le:EquivalenceFirstDerivative} we have that
	\begin{align*}
		& \quad \ph''(x, \abs{\nabla u} + \abs{\nabla u_{n_k}} ) \abs{\nabla u - \nabla u_{n_k}}^2 \\
		& \geq a_x^{-1} \ph''(x, \abs{\nabla u - \nabla u_{n_k}} ) \abs{\nabla u - \nabla u_{n_k}}^2 \\
		& \geq a_x^{-1} C \widetilde{C}_x \ph(x, \abs{\nabla u - \nabla u_{n_k}} ),
	\end{align*}
	which implies that $\nabla u_{n_k} \to \nabla u$ a.e. in that subset. On the other hand, in the complementary subset, by the limit assumption of \textnormal{(Mono2)}, we have that there exists $m(x) \geq 1$  such that $\abs{\nabla u_{n_k}} \leq m(x)$ for a.a.\,$x \in \Omega$. Thus, as $\ph''(x,\cdot)$ is almost decreasing,  
	\begin{align*}
		& \quad \ph''(x, \abs{\nabla u} + \abs{\nabla u_{n_k}} ) \abs{\nabla u - \nabla u_{n_k}}^2 \\
		& \geq a_x^{-1} \ph''(x, \abs{\nabla u} + m(x) ) \abs{\nabla u - \nabla u_{n_k}}^2, 
	\end{align*}
	and therefore $\nabla u_{n_k} \to \nabla u$ a.e. in this subset. Altogether, $\nabla u_{n_k} \to \nabla u$ a.e. in $\Omega$, and by the subsequence principle we can recover the whole sequence.
	
	Next, by Lemma \ref{Le:YoungIneqph}, we know
	\begin{align*}
		\langle A_\ph (u_n),u_n-u\rangle 
		& \geq \into \ph'(x, \abs{\nabla u_n} ) \abs{\nabla u_n} - \ph'(x, \abs{\nabla u_n} ) \abs{\nabla u} \dx \\
		& \geq \into \ph(x,\abs{\nabla u_n}) - \ph(x,\abs{\nabla u}) \dx
	\end{align*}
	which together with the limit superior assumption implies
	\begin{align*}
		\limsup_{n\to \infty} I_\ph (u_n) \leq I_\ph(u).
	\end{align*}
	On the other hand, by Fatou's Lemma for $\Phi$-functions (see, for example, \cite[Lemma 3.1.4]{Harjulehto-Hasto-2019}), we can obtain the reverse inequality with the limit inferior, and thus
	\begin{align*}
		\lim_{n\to \infty} I_\ph (u_n) = I_\ph(u).
	\end{align*}
	Using the claim from above, it is straightforward to also check that the integrand of the left-hand side converges in measure to the one in the right-hand side (pass to an a.e. convergent subsequence and apply the subsequence principle). The so-called converse of Vitali's theorem yields the convergence $\ph(\cdot,\abs{\nabla u_n}) \to \ph(\cdot,\abs{\nabla u})$ in $\Lp{1}$, so in particular the sequence $\ph(\cdot,\abs{\nabla u_n})$ is uniformly integrable. On the other hand, 
	\begin{align*}
		\ph(x,\abs{\nabla u_n - \nabla u}) 
		& \leq \ph(x,2 \max\{ \abs{\nabla u_n} , \abs{\nabla u} \}) \\
		& \leq K \ph(x,\max\{ \abs{\nabla u_n} , \abs{\nabla u} \}) \\
		& \leq K \left[ \ph(x,\abs{\nabla u_n} )  +  \ph(x,\abs{\nabla u} )  \right] ,
	\end{align*}
	and therefore the sequence $\ph(\cdot,\abs{\nabla u_n - \nabla u}) $ is uniformly integrable. Once again, it is straightforward to check that this sequence converges in measure to zero, so Vitali's theorem implies that 
	\begin{align*}
		\lim_{n\to \infty} \modph{ \nabla u_n - \nabla u }
		= \lim_{n\to \infty} \into \ph(x,\abs{\nabla u_n - \nabla u})  \dx 
		= 0.
	\end{align*}
	By the assumption \textnormal{(Poin)}, this means that $u_n \to u$ in $\Wphzero$.
\end{proof}

\begin{remark}
	Note that by \cite[Lemma 3.7.7]{Harjulehto-Hasto-2019}, if $\ph$ satisfies condition (A0) (see the definition in that book) together with \textnormal{(Mono1)}, then $\Lph \hookrightarrow \Lp{1}$.
\end{remark}

We also prove that the operator $A_\ph$ is bounded and coercive.

\begin{proposition}
	\label{Prop:BoundedCoercive}
	Let $\ph$ be a generalized N-function such that it satisfies \textnormal{(aInc)} and \textnormal{(aDec)}. Then $A_\ph$ is bounded. Furthermore, if $\ph$ satisfies \textnormal{(Poin)}, it is coercive.
\end{proposition}

\begin{proof}
	For the boundedness, consider $u, v \in \Wphzero$ with $\normoneph{v} \leq 1$. By Proposition \ref{Prop:AbstractHoelder}, Lemma \ref{Le:ConjugateInequality}, Proposition \ref{Prop:AbstractoneNormModular}, and Lemma \ref{Le:EquivalenceFirstDerivative}, we have that
	\begin{align*}
		\abs{\langle A_\ph(u) , v \rangle} 
		& \leq \into \ph'(x,\abs{\nabla u}) \abs{\nabla v} \dx \\
		& \leq \into \ph^*( \ph'(x,\abs{\nabla u}) ) \dx + \modph{ \nabla v } \\
		& \leq \into \abs{\nabla u} \ph'(x,\abs{\nabla u}) \dx + 1 \\
		& \leq C \modph{\nabla u} + 1.
	\end{align*}
	Once again using Proposition \ref{Prop:AbstractoneNormModular}, we deduce that
	\begin{align*}
		\norm{A(u)}_* \leq C \max \{ \normoneph{u}^{\widetilde{p}} , \normoneph{u}^{\widetilde{q}}\} + 1
	\end{align*}
	for some $1 < \widetilde{p} \leq \widetilde{q} < \infty$.
	
	For the coercivity, consider $u \in \Wphzero$ with $\normonephzero{u} = \normph{\nabla u} \geq 1$. From Lemma \ref{Le:EquivalenceFirstDerivative} and Proposition \ref{Prop:AbstractNormModular}, we derive
	\begin{align*}
		\frac{\langle A(u) , u \rangle}{ \normph{ \nabla u } } 
		& = \frac{1}{ \normph{ \nabla u } } \into \ph'(x, \abs{\nabla u}) \abs{\nabla u} \dx \\
		& \geq  \frac{C \modph{\nabla u}}{ \normph{ \nabla u } }\\
		& \geq C \normph{ \nabla u }^{\widetilde{p} - 1}
	\end{align*}
	for some $1 < \widetilde{p} < \infty$.
\end{proof}

Finally, in the following result we gather all the properties together and we further prove that $A_\ph$ is a homeomorphism.

\begin{theorem}
	Let $\ph$ be a generalized N-function that satisfies \textnormal{(Mono1)}, \textnormal{(Mono2)}, and \textnormal{(Poin)}. Then, $I_\ph$ is $C^1$, $I_\ph'=A_\ph$, $A_\ph$ satisfies the (S$_+$)-property, and it is strictly monotone, bounded, coercive, and a homeomorphism. 
\end{theorem}

\begin{proof}
	All the properties except that it is a homeomorphism were already proven in Proposition \ref{Prop:Differentiabilityph}, Proposition \ref{Prop:MonoAndSplus}, and Proposition \ref{Prop:BoundedCoercive}. The proof that it is a homeomorphism follows with standard arguments using the Minty-Browder theorem. This theorem can be found in \cite[Theorem 26.A]{Zeidler-1990}, and those standard arguments can be found in \cite[Theorem 3.3]{Crespo-Blanco-Gasinski-Harjulehto-Winkert-2022} or \cite[Theorem 4.4]{Arora-Crespo-Blanco-Winkert-2023}.
\end{proof}

For completion, statements on all usual variants of the operator $A_\ph$ are given. The arguments in these results are essentially the same as above, the most significant difference is whether the assumption \textnormal{(Poin)} is needed or not. 

\begin{definition}
	Let $\ph$ be a generalized convex $\Phi$-function. We define the operator $B_\ph \colon \Wph \to \left[ \Wph \right] ^*$ as 
	\begin{align*}
		\langle B_{\ph} (u) , v \rangle 
		& = \into a_\ph (x,\nabla u) \cdot \nabla v +  a_\ph (x, u) v \dx \\
		& = \into \ph'(x,\abs{ \nabla u }) \frac{ \nabla u }{\abs{ \nabla u }} \cdot \nabla v 
		+ \ph'(x,\abs{ u }) \frac{ u }{\abs{ u }}  v  \dx,
	\end{align*}
	and the functional $J_{\ph} \colon \Wph \to \R$ as
	\begin{align*}
		J_{\ph} (u) = \into \ph(x,\abs{ \nabla u }) + \ph(x,\abs{ u }) \dx.
	\end{align*}
	
	We further define $\widetilde{A_\ph} \colon \Wph \to \left[ \Wph \right] ^*$ and $\widetilde{I_{\ph}} \colon \Wph \to \R$ as
	\begin{align*}
		\widetilde{A_\ph} (u) = A_\ph(u), \qquad \qquad \widetilde{I_{\ph}} (u) = I_{\ph} (u),
	\end{align*}
	and $\widetilde{B_\ph} \colon \Wphzero \to \left[ \Wphzero \right] ^*$ and $\widetilde{J_{\ph}} \colon \Wphzero \to \R$ as
	\begin{align*}
		\widetilde{B_\ph} (u) = B_\ph(u), \qquad \qquad \widetilde{J_{\ph}} (u) = J_{\ph} (u).
	\end{align*}
\end{definition}

\begin{theorem}
	Let $\ph$ be a generalized N-function that satisfies \textnormal{(Mono1)}, \textnormal{(Mono2)}, and \textnormal{(Poin)}. Then, $\widetilde{I_\ph}$ is $C^1$, $\widetilde{I_\ph}'=\widetilde{A_\ph}$, $\widetilde{A_\ph}$ satisfies the (S$_+$)-property, and it is monotone and bounded.
\end{theorem}

\begin{proof}
	For the the proof of the (S$_+$)-property, use the compact embedding $\Wph$ $\hookrightarrow \Lph$ to obtain that $\lim_{n\to \infty} \modph{u_n - u} = 0$. The rest of the proofs are identical to those of $A_\ph$.
\end{proof}

\begin{remark}
	In order to recover the strongest properties for the operator $\widetilde{A_\ph}$, one should define it instead in the space of zero mean functions, that is
	\begin{align*}
		L^{\ph}_{\langle \cdot \rangle } (\Omega) & = \left\lbrace u \in \Lph : \into u \dx = 0 \right\rbrace, \\
		W^{1,\ph}_{\langle \cdot \rangle } (\Omega) & = \left\lbrace u \in \Wph : \into u \dx = 0 \right\rbrace.
	\end{align*}
	By the H\"older inequality (see Proposition \ref{Prop:AbstractHoelder}), the mean operator is continuous both in $\Lph$ and $\Wph$, hence $L^{\ph}_{\langle \cdot \rangle } (\Omega)$ and $W^{1,\ph}_{\langle \cdot \rangle } (\Omega)$ are closed subspaces, and thus reflexive, separable spaces. By (Poin), we can analogously prove that we have a Poincar\'e inequality in $W^{1,\ph}_{\langle \cdot \rangle } (\Omega)$. So, one can repeat without any change the argument of Theorem \ref{Th:GeneralVersion} and obtain that $\widehat{A_\ph} \colon W^{1,\ph}_{\langle \cdot \rangle } (\Omega) \to \left[ W^{1,\ph}_{\langle \cdot \rangle } (\Omega) \right] ^*$ and $\widehat{I_\ph} \colon W^{1,\ph}_{\langle \cdot \rangle } (\Omega) \to \R$, defined as $A_\ph$ and $I_\ph$, satisfy: $\widehat{I_\ph}$ is $C^1$, $\widehat{I_\ph}'=\widehat{A_\ph}$, $\widehat{A_\ph}$ satisfies the (S$_+$)-property, and it is strictly monotone, bounded, coercive, and a homeomorphism. 
\end{remark}

\begin{theorem}
	Let $\ph$ be a generalized N-function that satisfies \textnormal{(Mono1)} and \textnormal{(Mono2)}. Then, $J_\ph$ is $C^1$, $J_\ph'=B_\ph$, $B_\ph$ satisfies the (S$_+$)-property, and it is strictly monotone, bounded, coercive, and a homeomorphism; and, $\widetilde{J_\ph}$ is $C^1$, $\widetilde{J_\ph}'=\widetilde{B_\ph}$, $\widetilde{B_\ph}$ satisfies the (S$_+$)-property, and it is strictly monotone, bounded, coercive, and a homeomorphism. 
\end{theorem}

\begin{proof}
	Repeat the arguments made for $A_\ph$ separately for the part with gradients and the parts without gradients.
\end{proof}

The assumptions in the previous results can actually be further generalized. One can prove the following result and the corresponding cases with $\widetilde{A_\ph}$, $B_\ph$, and $\widetilde{B_\ph}$.

\begin{theorem} \label{Th:GeneralVersion}
	Let $\ph$ be a generalized N-function such that it satisfies \textnormal{(aInc)}, \textnormal{(aDec)}, $\ph(x,\cdot) \in C^1(0,\infty)$ for a.a.\,$x \in \Omega$, and \textnormal{(Poin)}.
	
	Further assume that $\ph=\ph_1 + \ph_2$, where $\ph_1$ is a generalized N-function that satisfies \textnormal{(Mono1)} and \textnormal{(Mono2)}, and $\ph_2$ satisfies for all $\eta, \xi \in \R^N$ and a.a. $x \in \Omega$,
	\begin{align*}
		( a_{\ph_2}(x,\eta) -  a_{\ph_2}(x,\xi) ) \cdot ( \eta - \xi ) \geq 0.
	\end{align*}
	
	Then, $I_\ph$ is $C^1$, $I_\ph'=A_\ph$, $A_\ph$ satisfies the (S$_+$)-property, and it is strictly monotone, bounded, coercive, and a homeomorphism. 
\end{theorem}

\begin{proof}
	All proofs are identical except the monotonicity and the (S$_+$)-property, in which we use the properties of $\ph_2$ to skip it and only work with $\ph_1$.
\end{proof}

As a closure for this work, we provide some examples of functions satisfying the assumptions of the theorems. First, here is a list of already studied generalized N-functions that still fit in this setting.

\begin{example} ~	
	\begin{enumerate}
		\item[{\textnormal{(a)}}] $\ph(x,t) = t^{p(x)}$, where $p \in C(\close)$ with $p(x)>1$ for all $x \in \close$, satisfies \textnormal{(Mono1)}, \textnormal{(Mono2)}, and \textnormal{(Poin)}.
		\item[\textnormal{(b)}] $\ph(x,t) = t^{p(x)} + \mu(x) t^{q(x)}$, where $p,q \in C(\close)$ with $1<p(x)\leq q(x)$ for all $x \in \close$ and $0 \leq \mu \in \Lp{1}$, satisfies \textnormal{(Mono1)}, but \textnormal{(Mono2)} and \textnormal{(Poin)} could be false. If $q(x) < p^*(x)$ for all $x \in \Omega$ and $\mu \in \Lp{\infty}$, then \textnormal{(Poin)} is also satisfied. \textnormal{(Mono2)} would require much stricter assumptions, like $2 \leq p(x)$ or $q(x) \leq 2$ for all $x \in \Omega$. However, we can avoid that by splitting $\ph$ into $\ph_1(x,t)=t^{p(x)}$ and $\ph_2(x,t) = \mu(x) t^{q(x)}$, and then the assumptions of Theorem \ref{Th:GeneralVersion} are satisfied and it yields the same conclusion. 
		\item[\textnormal{(c)}] $\ph(x,t) = t^{p(x)} + \mu(x) t^{q(x)} \log(e+t)$ is a very similar situation to \textnormal{(b)}.
	\end{enumerate}
\end{example}

We further provide one more noteworthy example, that is, one novel generalized N-function. As far as the author knows, the (S$_+$)-property had never been proven for this generalized N-function before.

\begin{example} ~	
	\begin{enumerate} 
		\item[\textnormal{(d)}] $\ph(x,t) = t^{p(x)}\log(e+t)$, where $p \in C(\close)$ with $p(x)>1$ for all $x \in \close$ satisfies satisfies \textnormal{(Mono1)}, \textnormal{(Mono2)}, and \textnormal{(Poin)}. Note that for \textnormal{(Mono2)}, $\ph''(x,\cdot)$ is increasing where $p(x) \geq 2$, decreasing where $p(x)$ is small enough, and almost decreasing in the rest.
	\end{enumerate}
\end{example}

\section*{Acknowledgments}
The author was funded by the Deutsche Forschungsgemeinschaft (DFG, German Research Foundation) under Germany's Excellence Strategy - The Berlin Mathematics Research Center MATH+ and the Berlin Mathematical School (BMS) (EXC-2046/1, project ID: 390685689).


\end{document}